\documentclass[a4paper,11pt]{article}
\pdfoutput=1

\usepackage{cite,hyperref}
\usepackage[active]{srcltx}
\usepackage{verbatim}
\usepackage{epsfig,graphicx,color,mathrsfs}
\usepackage{graphicx}
\usepackage{amsmath,amssymb,amsthm,amsfonts}
\usepackage{amssymb}
\usepackage{mathtools,xparse}
\usepackage[english]{babel}
\usepackage[usenames,dvipsnames,svgnames,table]{xcolor}
\usepackage{tikz}
\usepackage{caption}
\usepackage{float}
\usepackage{cite}
\usepackage{manyfoot}
\usepackage{xypic}

\usepackage[left=2.7cm,right=2.7cm,top=3cm,bottom=3cm]{geometry}

\hypersetup{
    colorlinks = true,
    citecolor = {blue},
    linkcolor = {blue},
    urlcolor = {red},
}

\newtheorem{theorem}{Theorem}
\newtheorem{proposition}{Proposition}

\newtheorem{corollary}{Corollary}
\newtheorem{remark}{Remark}
\newtheorem{definition}{Definition}

\renewcommand{\to}{\longrightarrow}

\renewcommand{\SS}{\mathrm{Sym^2}(C_2)}
\newcommand{\BB}{\mathsf {B}_2(C_2)}

\usepackage{epigraph}
\title{Monodromy representations and surfaces \\ with maximal Albanese dimension}
\author{Francesco Polizzi}
\date{}
\begin{document}
\maketitle

\epigraph{\itshape More and more I'm aware that the permutations are not unlimited.}{Russell Hoban, \textit{Turtle Diary}}


\begin{abstract}
We relate the existence of some surfaces of general type and maximal Albanese dimension to the existence of some monodromy representations of the braid group $\mathsf{B}_2(C_2)$ in the symmetric group $\mathsf{S}_n$. Furthermore, we compute the number of such representations up to $n=9$, and we analyze the cases $n \in \{2, \, 3, \, 4\}$. For $n=2, \, 3$ we recover some surfaces with $p_g=q=2$ recently studied (with different methods) by the author and his collaborators, whereas for $n=4$ we obtain some conjecturally new examples.  

\end{abstract}


\Footnotetext{{}}{\textit{2010 Mathematics Subject Classification}: 14J29, 20F36}

\Footnotetext{{}} {\textit{Keywords}: Braid group, monodromy representation, surface of general type}


\setcounter{section}{-1}

\section{Introduction} \label{sec:intro}

The classification of surfaces $S$ of general type with $\chi(\mathcal{O}_S)=1$, i.e. $p_g(S)=q(S)$, is currently an active area of research, see for instance the survey paper \cite{BaCaPi06}. For these surfaces, \cite[Th\'eor\`eme 6.1]{Deb82} implies $p_g \leq 4$, and the cases $p_g=q=4$ and $p_g=q=3$ are nowadays completely described, see \cite{CaCiML98}, \cite{HP02}, \cite{Pir02}. 

Regarding the case $p_g=q=2$, a complete classification has been recently obtained when $K_S^2=4$, see \cite{CMP14}, \cite{BPS16}. In fact, these are surfaces on the Severi line $K_S^2=4 \chi (\mathcal{O}_S)$; we refer the reader to the aforementioned papers and the references contained therein for a historical account on the subject and more details. 

By contrast, the classification in the case $p_g=q=2$, $K_S^2 \geq 5$ is still missing, albeit some interesting examples were recently discovered, see \cite{CH06}, \cite{PePol13a}, \cite{PePol13b}, \cite{PePol14}, \cite{PiPol16}, \cite{PolRiRo17}.

The purpose of this note is to show how monodromy representations of braid groups can be concretely applied to the fine classification of surfaces with $p_g=q=2$ and maximal Albanese dimension, allowing one to rediscover old examples and to find new ones.

The idea is to consider degree $n$, generic covers of $\SS$, the symmetric square of a smooth curve of genus $2$, simply branched over the diagonal $\delta$. In fact, if such a cover exists, then it is a smooth surface $S$ with
\begin{equation*}
\chi(\mathcal{O}_S)=1, \quad K_S^2=10-n,
\end{equation*} 
see Theorem \ref{thm:generic-cover}. Furthermore, if $p_g(S)=q(S)=2$ then the Albanese variety $\textrm{Alb}(S)$ is isogenous to the Jacobian variety $J(C_2)$ (Proposition \ref{prop:alb}) and, for a general choice of $C_2$, the surface $S$ contains no irrational pencils (Proposition \ref{prop:irrational-pencil}).

On the other hand, by the Grauert-Remmert extension theorem (see \cite{GrRem58}, \cite[XII.5.4]{SGA1}, \cite{DeGr94}) and the GAGA principle (see \cite{Serre56}, \cite[Chapter 6]{Serre08}), isomorphism classes of degree $n$, connected covers 
\begin{equation*}
f \colon S \to \SS, 
\end{equation*}
branched at most over $\delta$, correspond to group homomorphisms 
\begin{equation*}
\varphi \colon \pi_1(\SS - \delta) \to \mathsf{S}_n
\end{equation*}
with transitive image, up to conjugacy in $\mathsf{S}_n$. The group $\pi_1(\SS - \delta)$ is isomorphic to $\mathsf{B}_2(C_2)$, the braid group on two strings on $C_2$, whose presentation can be found for instance in \cite{Bel04}; furthermore, our condition that the branching is simple can be translated by requiring that $\varphi(\sigma)$ is a transposition, where $\sigma$ denotes the homotopy class in $\SS - \delta$ of a topological loop in $\SS$ that ``winds once" around $\delta$. 

A group homomorphism $\mathsf{B}_2(C_2) \to \mathsf{S}_n$ satisfying the requirements above will be called a \emph{generic monodromy representation} of $\mathsf{B}_2(C_2)$, see Definition \ref{def:generic-monodromy}. By using the Computer Algebra System \verb|GAP4| (see \cite{GAP4}) we computed the number of generic monodromy representations  for $2 \leq n \leq 9$, see Theorem \ref{thm:monodromy}. In particular, such a number is zero for $n \in \{5, \, 7, \, 9 \}$, so there exist no generic covers in these cases, see Corollary \ref{cor:no-surfaces}. For the reader's convenience, we included an Appendix containing the short script.        

The previous discussion can be now summarized as follows. 
\bigskip

$\textbf{Theorem}\; \,$\textit{Let $f \colon S \to \mathrm{Sym}^2(C_2)$ be a generic cover of degree $n$ and whose branch locus is the diagonal $\delta$. Then $S$ is a surface of maximal Albanese dimension with
$\chi(\mathcal{O}_S)=1$ and $K_S^2= 10-n$. Moreover, if $ 2 \leq n \leq 9$ then $S$ is of general type.}

\textit{The isomorphism classes of generic covers of degree} $n$  
 \textit{are in bijective correspondence to generic monodromy representations
$\varphi \colon \mathsf {B}_2(C_2) \to \mathsf {S}_n,$ up to conjugacy in $\mathsf {S}_n$. 
For $2 \leq n \leq 9$, the corresponding number of representations is given in the table below$:$}  
\begin{table}[H]
\begin{center}
\begin{tabular}{c|c|c|c|c|c|c|c|c}
$ n$ & $2$ & $3$ & $4$ & $5$ & $6$ & $7$ & $8$ & $9$ \\
 \hline
$\textrm{Number of}$ $\varphi$ & $16$ & $3 \cdot 80$ & $6 \cdot 480$ & $0$ & $15 \cdot 2880$  & $0$ & $28 \cdot 172800$ & $0$\\
\end{tabular}
\end{center}
\end{table}

Such results are still far for being conclusive, for at least two reasons:
\begin{itemize}
\item[$\boldsymbol{(1)}$] no attempt has been made here in order to compute the number of 
generic monodromy representations $\varphi \colon \mathsf{B}_2(C_2) \to \mathsf{S}_n$ for all values of $n$. In principle, our \verb|GAP4| script could do this, but in practice it is not efficient enough when $n$ is big (the computation in the case $n=9$ already took several hours). Furthermore, it would be desirable to extend our methods to non-generic representations, i.e. to non-generic covers of $\SS$;
\item[$\boldsymbol{(2)}$] given a generic cover $f \colon S \to \SS$, corresponding to a generic monodromy representation $\varphi \colon \mathsf{B}_2(C_2) \to \mathsf{S}_n$, it is at the moment not clear how to explicitly compute $K_{\bar{S}}^2$, where $\bar{S}$ denotes the minimal model of $S$: in fact, we know no general procedure to determine whether $S$ contains some $(-1)$-curves, see Proposition \ref{prop:minimality-S}. 
\end{itemize}
These are interesting problems that we hope to address in the future.
\smallskip 

Let us explain now how this paper is organized. In Section \ref{sec:prel} we collect some preliminary results 
that are needed in the sequel of the work, namely the Grauert-Remmert extension theorem, the GAGA principle and their corollaries, Bellingeri's presentation for $\mathsf{B}_2(C_2)$ and the classification of surfaces with $\chi(\mathcal{O}_S)=1$ and maximal Albanese dimension. In Section \ref{sec:monodromy} we prove our main results and we make a more detailed analysis of our covers in the cases $n=2, \, 3, \, 4$. It turns out that for $n=2$ and $n=3$ we rediscover some examples recently studied (using different methods) by the author and his collaborators, see \cite{PiPol16}, \cite{PolRiRo17}; on the other hand, for $n=4$ we conjecture that our construction provides new examples of minimal surfaces with $p_g=q=2$, $K^2=6$ and maximal Albanese dimension, that we plan to investigate in a sequel of this work.

\bigskip

\textbf{Notation and conventions.} We work over the field
$\mathbb{C}$ of complex numbers. 
By \emph{surface} we mean a projective, non-singular surface $S$,
and for such a surface $K_S$ denotes the canonical
class, $p_g(S)=h^0(S, \, K_S)$ is the \emph{geometric genus},
$q(S)=h^1(S, \, K_S)$ is the \emph{irregularity} and
$\chi(\mathcal{O}_S)=1-q(S)+p_g(S)$ is the \emph{Euler-Poincar\'e characteristic}.

We say that $S$ is \emph{of maximal Albanese dimension} if its Albanese map $a_S \colon S \to \mathrm{Alb}(S)$ is generically finite onto its image.

If $C$ is a smooth curve, we write $J(C)$ for the Jacobian variety of $C$.

The symbol $\mathsf {S}_n$ stands for the symmetric group on $n$ letters.

\section{Preliminaries} \label{sec:prel}

\subsection{Finite covers and monodromy representations} \label{subsec:covers-mon}

This subsection deals with the classification of branched covers 
$f \colon X \to Y$ of projective varieties via the classification of monodromy representations of the fundamental group $\pi_1(Y-B)$, where $B \subset Y$ is the branch locus of $f$. The main technical tools needed are the Grauert-Remmert extension theorem and the GAGA principle, that we recall below. 

\begin{proposition}[$\textbf{Grauert-Remmert extension theorem}$] \label{th:Grauert}
Let $Y$ be a normal analytic space and $Z \subset Y$ a closed analytic subspace such that $U=Y - Z$ is dense in $Y$. Then any finite, unramified cover $f^{\circ} \colon V \to U$ can be extended to a normal, finite cover $f \colon X \to Y$, and such an extension is unique up to isomorphisms.   
\end{proposition}
\begin{proof}
See \cite{GrRem58}, \cite[XII.5.4]{SGA1}, \cite{DeGr94}.
\end{proof}

\begin{proposition}[$\textbf{GAGA principle}$] 
\label{th:GAGA}
Let $X$, $Y$ be projective varieties over $\mathbb{C}$, and $X^{\rm an}$, $Y^{\rm an}$ the underlying complex analytic spaces. Then
\begin{itemize}
\item[$\boldsymbol{(1)}$] every analytic map $X^{\rm an} \to Y^{\rm an}$ is algebraic$;$
\item[$\boldsymbol{(2)}$] every coherent analytic sheaf on $X^{\rm an}$ is algebraic, and its algebraic cohomology coincides with its analytic one.
\end{itemize} 
\end{proposition}
\begin{proof}
See \cite{Serre56}, \cite[Chapter 6]{Serre08}.  
\end{proof}

The two results above imply the following fact concerning extensions of covers of quasi-projective varieties. With a slight abuse of notation, we write $X$ instead of $X^{\rm an}$. 

\begin{proposition} \label{prop:extension} 
Let $Y$ be a smooth, projective variety over $\mathbb{C}$ and $Z \subset Y$ be a smooth, irreducible divisor. Set $U=Y - Z$. Then any finite, unramified analytic cover $f^{\circ} \colon V \to U$ can be extended in a unique way to a finite cover $f \colon X \to Y,$ branched at most over $Z$. Moreover, there exists on $X$ a unique structure of smooth projective variety that makes $f$ an algebraic finite cover. 
\end{proposition}
\begin{proof}
By Proposition \ref{th:Grauert}, the cover $f^{\circ} \colon V \to U$ can be extended in a unique way to a finite analytic cover $f \colon X \to Y$. Such a cover corresponds to a coherent analytic sheaf of algebras over the projective variety $Y$; now Proposition \ref{th:GAGA} implies that such a sheaf is algebraic, hence so are $X$ and $f$. 

Since $X$ is normal and $Y$ is smooth, by the purity theorem (\cite[X.3.1]{SGA1}) the branch locus of $f$ is either empty or coincides with the smooth irreducible divisor $Z$. In both cases, a local computation shows that $X$ is a smooth scheme (see \cite[Lemma 2.1]{EdJS10}), then its underlying analytic space is a complex manifold, which is compact because it is a finite analytic cover of the compact manifold $Y$. 

Therefore $X$ is a smooth complete scheme endowed with a finite map $f \colon X \to Y$ onto the projective scheme $Y$. If $L$ is an ample line bundle on $Y$, by \cite[Proposition 1.2.13]{Laz04} it follows that $f^*L$ is an ample line bundle on $X$, so $X$ is a smooth projective variety and we are done.      
\end{proof}

\begin{corollary} \label{cor:monodromy-rep}
Let $Y$ be a smooth projective variety over $\mathbb{C}$ and $Z \subset Y$ be a smooth, irreducible divisor. Then isomorphism classes of connected covers of degree $n$
\begin{equation*}
f \colon X \to Y,
\end{equation*}
branched at most over $Z$, are in bijection to group homomorphisms with transitive image
\begin{equation} \label{eq:monodromy-rep}
\varphi \colon \pi_1(Y - Z) \to \mathsf{S}_n,
\end{equation}
up to conjugacy in $\mathsf{S}_n$. Furthermore, $f$ is a Galois cover if and only if the subgroup $\mathrm{im}\, \varphi$ of $\mathsf{S}_n$ has order $n$, and in this case $\mathrm{im}\, \varphi$ is isomorphic to the Galois group of $f$.
\end{corollary}
\begin{proof}
Set $U=Y-Z$. By \cite[Chapter 8]{SeiTh80} we know that isomorphism classes of degree $n$, connected topological 
covers $f^{\circ}\colon  V \to U$ are in bijection to conjugacy classes of group homomorphisms of type \eqref{eq:monodromy-rep}, and that Galois covers are precisely those such that $\mathrm{im} \, \varphi$ has order $n$. Since $U$ is a complex manifold, we can pull-back its complex structure to $V$, in such a way that $f^{\circ}$ becomes an analytic map. Then, by using Proposition \ref{prop:extension}, we can uniquely extend $f^{\circ}$ to a degree $n$, algebraic finite cover
$f \colon X \to Y$, branched at most over $Z$ and such that $X$ is smooth and projective. This completes the proof. 
\end{proof}
The group homomorphism $\varphi$ is called the \emph{monodromy representation} of the cover $f$, and its image $\mathrm{im}\, \varphi$ is called the $\emph{monodromy group}$ of $f$. By Corollary \ref{cor:monodromy-rep}, if $f$ is a Galois cover then the monodromy group of $f$ is isomorphic to its Galois group. 

\subsection{Braid groups on Riemann surfaces} For more details on the results of this subsection, we refer the reader to \cite{Bel04}.

Let $C_g$ be a compact Riemann surface of genus $g$ and $\mathscr{P} = \{p_1, \ldots, p_k\} \subset C_g$ a set of $k$ distinct points. A \emph{geometric braid} on $C_g$ based at $\mathscr{P}$ is a $k$-ple $(\psi_1, \ldots, \psi_k)$ of paths $\psi_i \colon [0, \, 1] \to C_g$ such that 
\begin{itemize}
\item $\psi_i(0) = p_i, \quad i=1, \ldots, k$;
\item $\psi_i(1) \in \mathscr{P}, \quad i=1, \ldots, k$;
\item the points $\psi_1(t), \ldots, \psi_k(t) \in C_g$ are pairwise distinct for all $t \in [0, \, 1]$.
\end{itemize}     

\begin{definition} \label{def:braid}
The \emph{braid group} on $k$ strings on $C_g$ is the group $\mathsf{B}_k(C_g)$ whose elements are the braids  based at $\mathscr{P}$ and whose operation is the usual product of paths, up to homotopies among braids. 
\end{definition}
It can be shown that $\mathsf{B}_k(C_g)$ does not depend on the choice of the set $\mathscr{P}$. Moreover, there is a group isomorphism 
\begin{equation} \label{eq:iso-braids}
\mathsf {B}_k(C_g) \simeq \pi_1(\textrm{Sym}^k(C_g) - \delta),
\end{equation}
where $\textrm{Sym}^k(C_g)$ denotes the $k$-th symmetric product of $C_g$, namely the quotient of the product $(C_g)^k$ by the natural permutation action of the symmetric group $\mathsf {S}_k$, and $\delta$ stands for the big diagonal in $\textrm{Sym}^k(C_g)$, namely the image of the set   
\begin{equation*}
\Delta= \{(x_1, \ldots, x_k) \, | \, x_i=x_j \; \textrm{for some} \; i \neq j \} \subset (C_g)^k. 
\end{equation*}
We are primarily interested in the case $g=k=2$.
\begin{proposition} \label{prop:generators-braid}
The braid group $\mathsf {B}_2(C_2)$ can be generated by five elements
\begin{equation*}
a_1, \, a_2, \, b_1, \, b_2, \, \sigma
\end{equation*}
subject to the eleven relations below$:$
\begin{equation*} \label{eq:relations}
\begin{split}
(R2) \quad &  \sigma^{-1} a_1 \sigma^{-1} a_1= a_1 \sigma^{-1} a_1 \sigma^{-1} \\ &  \sigma^{-1} a_2 \sigma^{-1} a_2= a_2 \sigma^{-1} a_2 \sigma^{-1} \\ &
\sigma^{-1} b_1 \sigma^{-1} b_1 = b_1 \sigma^{-1} b_1 \sigma^{-1} \\ & \sigma^{-1} b_2 \sigma^{-1} b_2 = b_2 \sigma^{-1} b_2 \sigma^{-1}\\ 
&  \\
(R3) \quad & \sigma^{-1} a_1 \sigma a_2 = a_2 \sigma^{-1} a_1 \sigma \\ &  \sigma^{-1} b_1 \sigma b_2 = b_2 \sigma^{-1} b_1 \sigma \\
& \sigma^{-1} a_1 \sigma b_2 = b_2 \sigma^{-1} a_1 \sigma \\
& \sigma^{-1} b_1 \sigma a_2 = a_2 \sigma^{-1} b_1 \sigma \\
 & \\
(R4) \quad & \sigma^{-1} a_1 \sigma^{-1} b_1 = b_1 \sigma^{-1} a_1 \sigma \\
 & \sigma^{-1} a_2 \sigma^{-1} b_2 = b_2 \sigma^{-1} a_2 \sigma \\
 & \\
 (TR) \quad &  [a_1, \, b_1^{-1}] [a_2, \, b_2^{-1}]= \sigma^2.  
\end{split}
\end{equation*}
\end{proposition}
\begin{proof}
See \cite[Theorem 1.2]{Bel04}, which provides a finite presentation for the general case $\mathsf {B}_k(C_g)$. 
\end{proof}
Geometrically speaking, the generators of $\mathsf {B}_2(C_2)$ in the statement of Proposition \ref{prop:generators-braid} can be interpreted as follows. The $a_i$ and the $b_i$ are the braids that come from the representation of the topological surface associated with $C_2$ as a polygon of $8$ sides with the standard identification of the edges, whereas $\sigma$ is the classical braid generator on the disk. In terms of the isomorphism \eqref{eq:iso-braids}, the generator $\sigma$ corresponds to the homotopy class in $\textrm{Sym}^2(C_2)-\delta$ of a topological loop that ``winds once around $\delta$".

\subsection{Surfaces of general type with $\chi(\mathcal{O}_S)=1$ and maximal Albanese dimension}

Let us describe now  surfaces with of general type with $p_g(S)=q(S)$ and maximal Albanese dimension.

\begin{proposition} \label{prop:class-surf-max-alb}
Let $S$ be a minimal surface of general type with $\chi (\mathcal{O}_S)=1$ and maximal Albanese dimension. Then we are in one of the  following situations$:$
\begin{itemize}
\item[$\boldsymbol{(1)}$] $p_g(S)=q(S)=4$, $K_S^2=8$ and $S=C_2 \times C_2'$, where $C_2$ and $C_2'$ are smooth curves of genus $2.$ In this case $\mathrm{Alb}(S) \simeq J(C_2) \times J(C_2')$ and  $a_S \colon S \to \mathrm{Alb}(S)$ is the product of the Abel-Jacobi maps of $C_2$ and $C_2'$, hence it is an immersion$;$
\item[$\boldsymbol{(2)}$] $p_g(S)=q(S)=3$, $K_S^2=6$ and $S=\mathrm{Sym}^2(C_3)$, where $C_3$ is a smooth curve of genus $3$. In this case $a_S \colon S  \to \mathrm{Alb}(S)$ is birational and its image is a principal polarization. More precisely, if $C_3$ is not hyperelliptic then $a_S$ is an immersion $($so its image is smooth$)$, whereas if $C_3$ is hyperelliptic then $a_S$ contracts the unique $(-2)$-curve on $S$ corresponding to the $g^1_2$ on $C_3$ $($so the image of $a_S$ has a rational double point of type $A_1);$ 
\item[$\boldsymbol{(3)}$] $p_g(S)=q(S)=3$, $K_S^2=8$ and $S=(C_2 \times C_3)/\mathbb{Z}_2$, where $C_2$ is a smooth curve of genus $2$ with an elliptic involution $\tau_2$, whereas $C_3$ is a smooth curve of genus $3$ with a free involution $\tau_3$ and the cyclic group $\mathbb{Z}_2$ acts freely on the product $C_2 \times C_3$ via the involution $\tau_2 \times \tau_3$. Setting 
\begin{equation*}
B=C_2 / \langle \tau_2 \rangle, \quad W=C_3/\langle \tau_3 \rangle
\end{equation*}
we have $g(B)=1$ and $g(W)=2$. Moreover,
the projections of $C_2 \times C_3$ onto $C_2$ and $C_3$ 
induce fibrations $b \colon S \to B$ and $w \colon S \to W$. The singular fibres of $b$ are two double fibres with smooth support, occurring at the branch points of $C_2 \to B$, whereas all the fibres of $w$ are smooth. Finally, $\mathrm{Alb}(S)$ is isogenous to $J(B) \times J(W);$ 
\item[$\boldsymbol{(4)}$] $p_g(S)=q(S)=2$ and $a_S \colon S \to \mathrm{Alb}(S)$ is a generically finite, branched cover.
\end{itemize}
\end{proposition}
\begin{proof}
The fact that $S$ is of maximal Albanese dimension implies $q(S) \geq 2$, so  \cite[Th\'eor\`eme 6.1]{Deb82} yields $p_g(S) \leq 4$. Thus we must only consider the cases $p_g(S)=q(S)=4$, $p_g(S)=q(S)=3$ and $p_g(S)=q(S)=2$. The first two possibilities have been classified in \cite{Be82}, \cite{CaCiML98}, \cite{HP02}, \cite{Pir02}, and they give cases $\boldsymbol{(1)}$,  $\boldsymbol{(2)}$,  $\boldsymbol{(3)}$; the last possibility gives case  $\boldsymbol{(4)}$.    
\end{proof}

\section{Monodromy representations of braid groups and surfaces with $p_g=q=2$} \label{sec:monodromy}

\subsection{Generic covers of $\mathrm{Sym}^2(C_2)$}

Let $C_2$ be a smooth curve of genus $2$ and let $\textrm{Sym}^2(C_2)$ be 
its second symmetric product. The Abel-Jacobi map
\begin{equation*}
\pi \colon \SS \to J(C_2)
\end{equation*}
is birational, more precisely it is the blow-down of the unique rational curve $E \subset \SS$, namely the $(-1)$-curve given by the unique $g^1_2$ on $C_2$. We have $\delta E=6$, because the curve $E$ intersects the diagonal $\delta$ transversally at the six points corresponding to the six Weierstrass points of $C_2$. Writing $\Theta$ for the numerical class of a theta divisor in $J(C_2)$, it follows that the image $D:=\pi_*\delta \subset J(C_2)$ is an irreducible curve with an ordinary sextuple point and no other singularities, whose numerical class is $4 \Theta$ (see \cite[Lemma 1.7]{PiPol16}).  

Using the terminology of \cite{MaPi02}, we can now give the following
\begin{definition} \label{def:generic-cover}
Let $f \colon S \to \SS$ be a connected cover of degree $n$
branched over the diagonal $\delta$, with ramification divisor $R \subset S$. Then $f$ is called \emph{generic} if  
\begin{equation*}
f^* \delta= 2R + R_0,
\end{equation*}
where the restriction $\left.f\right|_{R} \colon R \to \delta$ is an isomorphism and $R_0$ is an effective divisor over which $f$ is not ramified. 
\end{definition}
\bigskip 
Note that generic covers are never Galois, unless $n=2$ (in which case $f^* \delta = 2R$).
Since $\delta$ is smooth, the genericity condition in Definition \ref{def:generic-cover} is equivalent to requiring that the fibre of $f$ over any point of $\delta$ has cardinality $n-1$; thus the restriction morphism 
$\left.f\right|_{R_0} \colon R_0 \to \delta$ is a cover of degree $n-2$. Setting
\begin{equation*}
\Gamma=f^* \delta, \quad Z= f^*E,
\end{equation*}
we infer 
\begin{equation} \label{eq:intersections}
\Gamma^2= n \delta^2 = - 4 n, \quad Z^2 = n E^2 = -n, \quad \Gamma Z = n (\delta E) = 6n. 
\end{equation}
If we write 
\begin{equation*}
\alpha= \pi \circ f \colon S \to J(C_2), 
\end{equation*}
then $\alpha$ is a generically finite cover of degree 
$n$, simply branched over the smooth locus of $D$ and contracting the curve $Z$ to the unique singular point of $D$. The case where $Z$ is irreducible is illustrated in Figure \ref{fig:cover} below.

\begin{center} 
\begin{tikzpicture}[xscale=-1,yscale=-0.25,inner sep=0.7mm,place/.style={circle,draw=black!100,fill=black!100,thick}] 
\draw (-0.7,-0.5) rectangle (0.6,5);

\draw[red,rotate=92,x=6.28ex,y=1ex] (0.9,-0.85) cos (1,0) sin (1.25,1) cos (1.5,0) sin (1.75,-1) cos (2,0) sin (2.25,1) cos (2.5,0) sin (2.75,-1) cos (3,0) sin (3.25,1) cos (3.5,0) sin (3.6,-0.85);
\draw[rotate=92] (0.5,-0.025) .. controls (1.75,0.035) and (2.75,0.035) .. (4,-0.025);

\draw (-6.7,-0.5) rectangle (-5.4,5);

\draw[red,rotate=92,x=6.28ex,y=1ex,xshift=6,yshift=170] (0.9,-0.85) cos (1,0) sin (1.25,1) cos (1.5,0) sin (1.75,-1) cos (2,0) sin (2.25,1) cos (2.5,0) sin (2.75,-1) cos (3,0) sin (3.25,1) cos (3.5,0) sin (3.6,-0.85);
\draw[rotate=92,xshift=6,yshift=170] (0.5,-0.025) .. controls (1.75,0.035) and (2.75,0.035) .. (4,-0.025);

\draw (-6.7,15.5) rectangle (-5.4,21);

\draw[red,xscale=-1,yscale=-4] (6.05,-4.55) node(A0) [place,scale=0.2]{} to [in=5,out=55,looseness=8mm,loop] () to [in=65,out=115,looseness=8mm,loop] () to [in=125,out=175,looseness=8mm,loop] () to [in=185,out=235,looseness=8mm,loop] () to [in=245,out=295,looseness=8mm,loop] () to [in=305,out=355,looseness=8mm,loop] ();

\draw[xscale=-1,yscale=-4] (0.2,-0.1) node(B0) []{\textcolor{red}{\footnotesize{$R$}}};
\draw[xscale=-1,yscale=-4] (0.25,-1.05) node(0B) []{\footnotesize{$Z$}};

\draw[xscale=-1,yscale=-4] (6.25,-0.1) node(B1) []{\textcolor{red}{\footnotesize{$\delta$}}};
\draw[xscale=-1,yscale=-4] (6.25,-1.05) node(1B) []{\footnotesize{$E$}};

\draw[xscale=-1,yscale=-4] (6.4,-5) node(B2) []{\textcolor{red}{\footnotesize{$D$}}};

\draw[xscale=-1,yscale=-4] (-0.9,-0.125) node(C0) []{$S$};
\draw[xscale=-1,yscale=-4] (7,-0.125) node(C1) []{$\quad \quad \quad \;  \SS$};
\draw[xscale=-1,yscale=-4] (7,-5.00) node(C2) []{$\quad \quad J(C_2)$};

\draw[->] (-0.9,2.25) -- (-5.2,2.25) node[midway,above] {$f$};

\draw[->] (-6.05,5.8) -- (-6.05,14.7) node[midway,right] {$\pi$};

\draw[->] (-0.9,4.8) -- (-5.2,15.5) node[midway,below=3pt] {$\alpha$}; 
\end{tikzpicture}

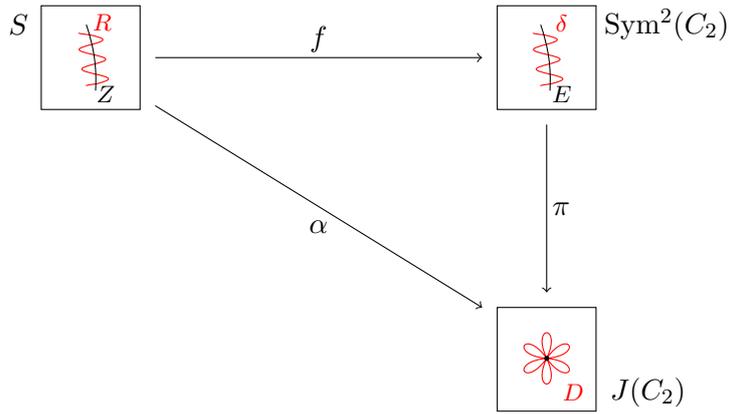
\captionof{figure}{The triple covers $f$ and $\alpha$}
\label{fig:cover} 
\end{center}

\begin{theorem} \label{thm:generic-cover}
Let $f \colon S \to \mathrm{Sym}^2(C_2)$ be a generic cover of degree $n$ and whose branch locus is the diagonal $\delta$. Then $S$ is a surface of maximal Albanese dimension with
\begin{equation*}
\chi(\mathcal{O}_S)=1, \quad K_S^2= 10-n. 
\end{equation*}
Moreover, if $ 2 \leq n \leq 9$ then $S$ is of general type.
\end{theorem}
\begin{proof}
The canonical class of $S$ is given by
\begin{equation} \label{eq:can-S}
K_S = f^* K_{\SS} + R = f^*E + R = Z + R. 
\end{equation}
The curve $R$ is smooth of genus $2$, so the genus formula and \eqref{eq:can-S} yield
\begin{equation*} 
2=2g(R)-2=R(R+K_S)=R(2R+Z)=2R^2+RZ.
\end{equation*}
On the other hand, by using the projection formula we can write
\begin{equation*}
RZ = R \cdot f^*E = f_*R \cdot E = \delta E = 6, 
\end{equation*}
hence $R^2=-2$. Thus can find $K_S^2$, in fact
\begin{equation} \label{eq:K2}
K_S^2 = (Z+R)^2 = Z^2 + 2RZ + R^2= -n + 12 - 2 = 10-n.
\end{equation}
Now we have to compute $\chi(\mathcal{O}_S)$. Squaring both sides of $2 R + R_0=\Gamma$ yields
\begin{equation} \label{eq:R-1}
4 R R_0 + (R_0)^2= -4 R^2 + \Gamma^2= 8-4n.
\end{equation}
Moreover, again by the projection formula we infer
\begin{equation} \label{eq:R-2}
\Gamma R_0 = f^* \delta \cdot R_0 = \delta \cdot f_* R_0 = \delta \cdot (n-2) \delta = 8-4n. 
\end{equation}
Combining \eqref{eq:R-1} with \eqref{eq:R-2}, we get
\begin{equation*}
0 = 4RR_0+(R_0)^2- \Gamma R_0 = (4R+R_0 - \Gamma)R_0 = 2RR_0. 
\end{equation*}
Therefore the effective curves $R$ and $R_0$ are disjoint. This in turn allows us to compute $c_2(S)$; in fact, writing $\chi_{\rm top}$ for the topological Euler number and recalling that $\left.f\right|_{R_0} \colon R_0 \to \delta$ is \'etale, by additivity we obtain
\begin{equation} \label{eq:c2}
\begin{split}
c_2(S) =\chi_{\rm top}(S) & = \chi_{\rm top}(S - R - R_0) + \chi_{\rm top}(R) + \chi_{\rm top}(R_0)\\
& = n  \cdot \chi_{\rm top}(\SS - \delta) + (-2) + (-2)(n-2) \\
& = 3n - 2 - 2(n-2)=  n+2.
\end{split}
\end{equation}
By using \eqref{eq:K2} and \eqref{eq:c2} together with Noether formula, we get $\chi(\mathcal{O}_S)=1$. 

The surface $S$ is of maximal Albanese dimension because, by the universal property of the Albanese map, the surjective morphism $\alpha \colon S \to J(C_2)$ factors through $a_S \colon S \to \mathrm{Alb}(S)$, so the image of $a_S$ has dimension $2$. In particular we have $q(S) \geq 2$. If $2 \leq n \leq 9$ then $S$ is irregular with $K_S^2>0$, hence of general type by \cite[Proposition X.1]{Be82}.
\end{proof}
The values of $n \in \{2, \ldots, 9\}$ for which generic covers do exist will be given in Theorem \ref{thm:monodromy}. We have at the moment no general method to determine whether the surface $S$ described in Theorem \ref{thm:generic-cover} is minimal or not. A partial result about locating its exceptional curves is the following

\begin{proposition} \label{prop:minimality-S}
The curve $Z=f^*E$ is reducible for $n >4$. Moreover, all $(-1)$-curves of $S$, if any, are components of $Z$. 
\end{proposition}
\begin{proof}
By computing the arithmetic genus of $Z$, we obtain  
\begin{equation}
p_a(Z)=\frac{Z(Z+K_S)}{2}+1 = \frac{Z(2Z+R)}{2}+1 = -n+4.
\end{equation}
For $n >4$ this quantity is negative, hence 
$Z$ is reducible and the first claim follows. The second claim is an immediate consequence of the fact that the only rational curve in $\mathrm{Sym}^2(C_2)$ is $E$.
\end{proof}

\begin{remark}
In the cases $n=2$ and $n=3$ the curve $Z$ is actually irreducible and $S$ is minimal, see \emph{Subsections \ref{subsec:n=2}, \ref{subsec:n=3}}. The irreducibility of $Z$ for $n=4$ is still an open problem, see \emph{Subsection \ref{subsec:n=4}}.
\end{remark}

Let us consider now the case $q(S)=2$.

\begin{proposition} \label{prop:alb}
Let $S$ be as in $\mathrm{Theorem}$ $\mathrm{\ref{thm:generic-cover}}$, and assume in addition that $q(S)=2$. Then $\mathrm{Alb}(S)$ is isogenous to $J(C_2)$, more precisely there exists an isogeny $\beta \colon \mathrm{Alb}(S) \to J(C_2)$ such that $\alpha = \beta \circ a_S$. In particular, if $n$ is prime then $a_S$ coincides with $\alpha$, up to automorphisms of $J(C_2)$.
\end{proposition}
\begin{proof}
Since $q(S)=2$, the Albanese variety $\mathrm{Alb}(S)$ is an abelian surface, so $a_S \colon S \to \mathrm{Alb}(S)$ is generically finite and, by the universal property, there is an isogeny $\beta \colon \mathrm{Alb}(S) \to J(C_2)$ such that the following diagram commutes:
\begin{equation} \label{dia.albanese}
\begin{split}
\xymatrix{
S \ar[r]^-{a_S} \ar[dr]_{\alpha} & \mathrm{Alb}(S) \ar[d]^{\beta} \\
 & J(C_2).}
\end{split}
\end{equation}
In particular, $\deg \beta$ divides $n$. If $n$ is prime, since $S$ is not birational to an abelian surface (recall that $\chi(\mathcal{O}_S)=1$) we get $\deg \beta =1$; this means that $\beta$ is a birational morphism between abelian surfaces, hence an isomorphism.   
\end{proof}

Recall that an \emph{irrational pencil} (or \emph{irrational fibration}) 
on a smooth, projective surface is a surjective morphism with connected fibres over a curve of positive genus.
\begin{proposition} \label{prop:irrational-pencil}
Let $S$ be as in $\mathrm{Theorem}$ $\mathrm{\ref{thm:generic-cover}}$ and assume that $q(S)=2$. If $\phi \colon S \to W$ is an irrational pencil on $S$, then $g(W)=1$. Moreover, 
the general surface $S$ contains no irrational pencils at all. 
\end{proposition}
\begin{proof}
We borrow the following argument from \cite[Proposition 1.9]{PiPol16}. 
Since $q(S)=2$, we have either $g(W)=1$ or $g(W)=2$. The latter case must be excluded: otherwise, using the embedding $W \hookrightarrow J(W)$ and the universal property, we would obtain a morphism of abelian surfaces $\mathrm{Alb}(S) \to J(W)$ with image isomorphic to the genus $2$ curve $W$, contradiction. Then $g(W)=1$ and $\mathrm{Alb}(S)$ must be a non-simple abelian surface. On the other hand, by Proposition \ref{prop:alb} we know that $\mathrm{Alb}(S)$ is isogenous to $J(C_2)$, and the latter surface is simple for a general choice of the curve $C_2$, see \cite[Theorem 3.1]{Ko76}. So, for a general choice of $S$, the Albanese variety $\mathrm{Alb}(S)$ is also simple and there are no irrational pencils on $S$. 
\end{proof}

We are now ready to apply the theory developed in Subsection \ref{subsec:covers-mon} in order to produce generic covers $f \colon S \to \SS$. 
\begin{definition} \label{def:generic-monodromy}
A \emph{generic monodromy representation} of the braid group $\mathsf {B}_2(C_2)$ is a group homomorphism 
\begin{equation*}
\varphi \colon \BB \to \mathsf{S}_n
\end{equation*}
with transitive image and such that $\varphi(\sigma)$ is a transposition.
\end{definition}
Generic covers and generic monodromy representations are related by the following

\begin{theorem} \label{thm:monodromy}
Isomorphism classes of generic covers of degree $n$ 
\begin{equation*}
f \colon S \to \SS,
\end{equation*}
with branched locus $\delta$, are in bijective correspondence to generic monodromy representations
\begin{equation*}
\varphi \colon \mathsf {B}_2(C_2) \to \mathsf {S}_n,
\end{equation*}    
up to conjugacy in $\mathsf {S}_n$. 
For $2 \leq n \leq 9$, the number of such representations is given in the table below$:$  
\begin{table}[H]
\begin{center}
\begin{tabular}{c|c|c|c|c|c|c|c|c}
$ n$ & $2$ & $3$ & $4$ & $5$ & $6$ & $7$ & $8$ & $9$ \\
 \hline
$\textrm{Number of}$ $\varphi$ & $16$ & $3 \cdot 80$ & $6 \cdot 480$ & $0$ & $15 \cdot 2880$  & $0$ & $28 \cdot 172800$ & $0$\\
\end{tabular}
\end{center}
\end{table}
\end{theorem}
\begin{proof}
The first part of the statement is an immediate consequence of Corollary \ref{cor:monodromy-rep} and isomorphism \eqref{eq:iso-braids}. The computation of number of monodromy representations with $\varphi(\sigma)=(1 \, 2)$  was done by using a short \verb|GAP4| script, that the reader can find in the Appendix. The total number of representations is obtained by multiplying such a number by the number of transpositions in $\mathsf{S}_n$, which is $n(n-1)/2$. 
\end{proof}

As an immediate consequence of Theorem \ref{thm:monodromy}, we can now state the following non-existence result.
\begin{corollary} \label{cor:no-surfaces}
Let $n \in \{5, \, 7, \, 9 \}$. Then there exist no surfaces with $p_g=q=2$ whose Albanese map is a generically finite, degree $n$ cover of  $J(C_2)$ simply branched over the smooth locus of the curve $D$.
\end{corollary}
\begin{proof}
If $n \in \{5, \, 7, \, 9 \}$, by Theorem \ref{thm:monodromy} there are no generic covers $f \colon S \to \SS$ of degree $n$ and  branched over $\delta$. 
\end{proof}

Finally, let us describe the situation in more details when $n \in \{2, \, 3, \, 4 \}$. 
\subsection{The case $n=2$} \label{subsec:n=2}
In this case we are looking for generic monodromy representations
\begin{equation*}
\varphi \colon \mathsf {B}_2(C_2) \to \mathsf {S}_2 = \{(1), \, (1 \, 2)\}.
\end{equation*} 
Since $\mathsf{B}_2(C_2)$ is generated by five elements $a_1$, $a_2$, $b_1$, $b_2$, $\sigma$ and necessarily $\varphi(\sigma)=(1 \, 2)$, we immediately see that there are $2^4=16$ possibilities for $\varphi$. The group $\mathsf{S}_2$ is abelian, so there is no conjugacy relation to consider and we get sixteen isomorphism classes of double covers $f \colon S \to \SS$, branched over $\delta$ and with 
\begin{equation*}
\chi(\mathcal{O}_S)=1, \quad K_S^2=8.
\end{equation*}
These covers correspond to the sixteen square roots of $\delta$ in  the Picard group of $\SS$; all of them give minimal surfaces by Proposition \ref{prop:minimality-S}, since $Z$ is a smooth, irreducible curve of genus $2$.  One cover coincides with the natural projection $f \colon C_2 \times C_2 \to \SS$, in fact 
\begin{equation*}
p_g(C_2 \times C_2) = q(C_2 \times C_2) =4, \quad K_{C_2 \times C_2}=8.
\end{equation*}
We claim that the remaining fifteen covers are surfaces with
\begin{equation*}
p_g(S)=q(S)=2, \quad K_S^2=8.
\end{equation*}
Indeed, otherwise, $p_g(S)=q(S)=3$ and $S$ would belong to case $\boldsymbol{(3)}$ of Proposition \ref{prop:class-surf-max-alb}, in particular it would admit an irrational pencil $\phi \colon S \to W$ with $g(W)=2$, contradicting Proposition \ref{prop:irrational-pencil}.  By Proposition \ref{prop:alb}, the  Albanese map of $S$ is generically finite of degree $2$ onto $J(C_2)$. 

These surfaces are studied in \cite[Section 2]{PolRiRo17}.

\subsection{The case $n=3$} \label{subsec:n=3}
In this case we are looking for generic monodromy representations
\begin{equation*}
\varphi \colon \mathsf{B}_2(C_2) \to \mathsf{S}_3, 
\end{equation*} 
up to conjugacy in $\mathsf{S}_3$. 

The output of our \verb|GAP4| script shows that if $\varphi(\sigma)=(1 \, 2)$ there are $80$ different choices for $\varphi$,  so the total  number of monodromy representations is $3 \cdot 80 =240$. For every such a representation we have $\mathrm{im}\, \varphi = \mathsf{S}_3$.

The \verb|GAP4| script also shows that each orbit for the conjugacy action of $\mathsf{S}_3$ on the set of monodromy representations consists of six elements, and consequently the orbit set has cardinality $240/6 = 40$. 

By Theorem 	\ref{thm:monodromy}, this implies that there are $40$ isomorphism classes of generic covers $f \colon S \to \SS$ of degree $3$ and branched over $\delta$. For all of them, the surface $S$ satisfies
\begin{equation*}
  p_g(S)=q(S)=2, \quad K_S^2=7 
\end{equation*}
and, by Proposition \ref{prop:alb}, its Albanese map is a generically finite cover of degree $3$ onto $J(C_2)$. These surfaces were studied in \cite{PiPol16}, where it is proved, with different methods, that they are all minimal (it turns out that $Z$ is a smooth, irreducible curve of genus $1$) and lie in the same deformation class. In fact, their moduli space is a connected, quasi-finite cover of degree $40$ of $\mathcal{M}_2$, the coarse moduli space of curves of genus $2$.

\subsection{The case $n=4$} \label{subsec:n=4}
In this case we are looking for generic monodromy representations
\begin{equation*}
\varphi \colon \mathsf{B}_2(C_2) \to \mathsf{S}_4, 
\end{equation*} 
up to conjugacy in $\mathsf{S}_4$.  

The output of our \verb|GAP4| script shows that if $\varphi(\sigma)=(1 \, 2)$ there are $480$ different choices for $\varphi$,  so the total  number of monodromy representations is $6 \cdot 480 =2880$. For every such a representation we have $\textrm{im} \, \varphi \simeq \mathsf{D}_8$, the dihedral group of order $8$. 

The \verb|GAP4| script also shows that each orbit for the conjugacy action of $\mathsf{S}_4$ on the set of monodromy representations consists of $12$ elements, and consequently the orbit set has cardinality $2880/12 = 240$.  

By Theorem 	\ref{thm:monodromy}, this implies that there are $240$ isomorphism classes of generic covers $f \colon S \to \SS$ of degree $4$ and branched over $\delta$. For all of them, the surface $S$ satisfies
\begin{equation*}
 \chi(\mathcal{O}_S)=1, \quad K_S^2=6. 
\end{equation*}
We do not know whether the curve $Z$ is irreducible or not. However, we conjecture that, at least for some of these covers, $S$ is a minimal model with $p_g(S)=q(S)=2$, and this would provide new examples of surfaces with these invariants and maximal Albanese dimension. We will not try to develop this point here, planning to come back to the problem in a sequel of this paper. 

\bigskip

$\mathbf{Acknowledgements.}$ The author was partially supported by GNSAGA-INdAM.

He thanks the organizers of the workshop \emph{Birational Geometry of Surfaces} (University of Rome Tor Vergata, January 2016) for the invitation and the hospitality.   

He is also indebted with ``abx", ``aglearner", Ariyan Javanpeykar, Stefan Behrens and Mohan Ramachandran for interesting discussions on several \verb|MathOverflow| threads, with R. Pardini and V. Coti Zelati for their support during the editing process and with the anonymous referee for helpful comments and remarks.

\section*{Appendix: the GAP4 script}

This short appendix contains the \verb|GAP4| script used in the paper. We explicitly write down the version for $n=3$. For the other cases, it suffices to change the first line \verb|n=3| to the desired value of $n$. 
\begin{verbatim}
 n:=3;;
 G:=SymmetricGroup(n);; t:=0;;
 s:=(1,2);;
 R:=[];; 
 for a1 in G do
 Ga1:=Centralizer(G, s*a1*s); 
 for b1 in Ga1 do
 Gb1:=Centralizer(G, s*b1*s);
 Ga1b1:=Intersection(Ga1, Gb1);
 for a2 in Ga1b1 do
 Ga2:=Centralizer(G, s*a2*s);
 Ga1b1a2:=Intersection(Ga1b1, Ga2);
 for b2 in Ga1b1a2 do
 H:=Subgroup(G, [s, a1, a2, b1, b2]); 
 R21 := s^(-1)*a1*s^(-1)*a1*(a1*s^(-1)*a1*s^(-1))^(-1);
 R22 := s^(-1)*a2*s^(-1)*a2*(a2*s^(-1)*a2*s^(-1))^(-1);
 R23 := s^(-1)*b1*s^(-1)*b1*(b1*s^(-1)*b1*s^(-1))^(-1);
 R24 := s^(-1)*b2*s^(-1)*b2*(b2*s^(-1)*b2*s^(-1))^(-1);
 TR := a1*b1^(-1)*a1^(-1)*b1*a2*b2^(-1)*a2^(-1)*b2*s^(-2);
 if  IsTransitive(H, [1..n])=true and 
 Order(R21)=1 and 
 Order(R22)=1 and 
 Order(R23)=1 and 
 Order(R24)=1 and 
 Order(TR)=1 then
 AddSet(R, [s, a1, b1, a2, b2]); 
 t:=t+1; Print(IdSmallGroup(H), " "); Print(t, "\n"); 
 fi; od; od; od; od;
 Size(R);
 H:=Centralizer(G, (1, 2));;
 T:=OrbitsDomain(H, R, OnTuples);;
 Size(T); 
 \end{verbatim}
\bigskip


\bigskip
\bigskip

Francesco Polizzi \\ Dipartimento di Matematica e Informatica,
Universit\`{a} della
Calabria \\ Cubo 30B, 87036 Arcavacata di Rende (Cosenza), Italy.\\
\emph{E-mail address:} \verb|polizzi@mat.unical.it|

\end{document}